\documentclass{birkjour}
\usepackage{amssymb}
\usepackage{amsthm}
\usepackage{newlfont}
\usepackage{amsfonts}
\usepackage{graphicx}
\usepackage[all]{xy}
\usepackage{subfigure}

\newtheorem{theorem}{Theorem}[section]
\newtheorem{lemma}[theorem]{Lemma}

\theoremstyle{definition}

\newtheorem{proposition}[theorem]{Proposition}

\theoremstyle{remark}

\numberwithin{equation}{section}

\newcommand{\dif}{\mathrm{d}}

\newcommand{\e}[1]{\mathbf{e}_{#1}}
\newcommand{\te}[1]{\widetilde{\mathbf{e}}_{#1}}
\newcommand{\tte}[1]{\widetilde{\widetilde{\mathbf{e}}}_{#1}}

\DeclareMathOperator{\B}{\mathbf{B}}
\DeclareMathOperator{\HH}{\mathbf{H}}
\DeclareMathOperator{\K}{K}
\DeclareMathOperator{\SF}{S}
\DeclareMathOperator{\KN}{K_N}
\DeclareMathOperator{\rank}{rank}

\begin{document}

%%%%%%%%   TÍTULO   %%%%%%%%%%%

\title[Minimal hypersurfaces in $\mathbb{S}^5$]
{Minimal hypersurfaces in $\mathbb{S}^5$ with vanishing Gauss-Kronecker curvature}

\author{Marcos M. Diniz}
\address{Universidade Federal do Par\'a\\ Instituto de Ci\^encias Exatas e Naturais}
\email{mdiniz@ufpa.br}
\thanks{CAPES - PROCAD--NF/2010}

\author{Jos\'{e} Antonio M. Vilhena}
\address{Universidade Federal do Par\'a\\ Instituto de Ci\^encias Exatas e Naturais}
\email{vilhena@ufpa.br}
\thanks{CAPES - PROCAD--NF/2010}

\author{Juan F. Z. Zapata}
\address{Universidade de S\~ao Paulo\\ Instituto de Matem\'atica e Estat\'istica}
\email{jfzzmat@ime.usp.br}
\thanks{CNPq - BOLSA   {\flushright{\bfseries \large August 02, 2012}}}

%\subjclass is required.
%    The 2010 edition of the Mathematics Subject Classification is
%    now available.  If you are citing a classification from the
%    new scheme, use the following input coding instead.

\date{August 02, 2012}

%\date{}

\begin{abstract}
In this paper we present a local description for complete minimal hypersurfaces in $\mathbb{S}^5$ 
with zero Gauss-Kronecker curvature, zero $3$-mean curvature and nowhere zero second fundamental form. 
\end{abstract}

\maketitle

%%%%%%%%%%%%%%%%%%%%%%%%%%%%%%%%%%%%%%%%%%%%%%%%%%%%%%%%%%%%%

\section*{Introduction}
Let $\mathbf{g}\! : M^2\hookrightarrow \mathbb{S}^{4}$ be a minimal immersion, where $M^2$ is a 2-dimensional manifold. Dajczer and Gromoll in \cite{Djaczer.1985} proved that if $\mathbf{g}$ has nowhere vanishing normal curvature, then the polar map of $\mathbf{g}$ is everywhere regular and provides a minimal hypersurface in $\mathbb{S}^{4}$ with Gauss-Kronecker curvature identically zero. In \cite{Brito.1987} de Almeida and Brito classified  the compact minimal hypersurfaces $M^3$ in $\mathbb{S}^{4}$ with zero Gauss-Kronecker curvature and non-vanishing second fundamental form on $M^3.$ Later, Ramanathan \cite{Ramanathan.1990} classified  the compact minimal hypersurfaces $M^3$ in $\mathbb{S}^{4}$ with zero Gauss-Kronecker curvature without the condition on the second fundamental form. In \cite{Hasanis.2007} Hasanis, Halilaj and Vlachos classified the complete minimal imersions $\mathbf{g}\! : M^2\hookrightarrow \mathbb{S}^{4}$ with Gauss-Kronecker curvature identically zero under some assumptions on the 
second fundamental form. In \cite{Asperti.2011} Asperti, Chaves and Sousa Jr. showed that the infimum of the absolute value of the Gauss-Kronecker curvature of a complete minimal hypersurfaces in 4-dimensional space form vanishes under assumptions on the Ricci curvature and classified the complete minimal hypersurfaces in 4-dimensional space form with Gauss-Kronecker curvature constant. 

In this paper, we consider $\mathbf{g}\! : M^2\hookrightarrow \mathbb{S}^{5}$ be a minimal immersion, where $M^2$ is a complete 2-dimensional manifold. Denote by $\K$ and $\KN$ (see \eqref{eq:Normal_Curvature}) the Gauss and scalar normal curvature, respectively. In Proposition \ref{Proposition1}  we show that $\mathbf{g}$ is superminimal if and only if   
\begin{equation*}
(\K-1)^2-\frac{1}{4}\KN=0.
\end{equation*}

In Theorem \ref{Teorema1} we give an example of a minimal hypersurface in $\mathbb{S}^5$ with vanishing Gauss-Kronecker  and 3-mean curvatures. This example is  a 2-spherical local bundle over a minimal surface in $\mathbb{S}^5$. Finally in Theorem \ref{Teorema2} we study locally the complete minimal hypersurfaces in $\mathbb{S}^5$ with zero Gauss-Kronecker curvature, zero $3$-mean curvature and nowhere zero second fundamental form.

\section{Preliminaries}

\subsection*{Isometric immersions in Euclidean sphere}
Let $\mathbf{g}\! : M^n\hookrightarrow \mathbb{S}^{n+p}$ an isometric immersion from an $n$-manifold $M^n$ into the $(n+p)$-sphere. Let $\mathcal{B} =\{ \e1, \e2, \ldots, \e{n+p}\}$ be an orthonormal frame on $\mathbb{S}^{n+p}$ \emph{adapted} to the immersion, in the sense that $\{\e1,\ldots,\e{n}\}$ spans $\mathbf{g}_*(TM)$, while $\{\e{n+1},\ldots,\e{n+p}\}$ spans $(\mathbf{g}_*(TM))^\bot$. Let  $\mathcal{B^*} =\{ \omega^1,\omega^2,\ldots,\omega^{n+p}\}$ its dual co-frame. The structure equations of $g(M^n)$, in terms of this co-frame are given by
\begin{equation}\label{eq:structure}
  \left\{
    \begin{aligned}
      \dif \mathbf{g}   &= \omega^i \otimes \e{i} ;\\
      \dif \e{i} &= - \omega^i \otimes \mathbf{g} +\omega^j_i\otimes \e{j}  				+\omega^\alpha_i \otimes \e{\alpha} ;\\
      \dif \e{\alpha} &= \omega^j_\alpha\otimes \e{j} +\omega^\beta_\alpha \otimes \e{\beta} ;\\
      \dif \omega^i &= - \omega^i_j \wedge \omega^j ,\quad 	\omega^i_j+\omega^j_i=0 ;\\
      \dif \omega^\alpha &= - \omega^\alpha_i \wedge \omega^i ,
    \end{aligned}
  \right.
\end{equation}
where $1\leqslant i,j\leqslant n$ and $n+1\leqslant \alpha, \beta \leqslant n+p$. Throughout all this work, we shall use small Latin letters, $i,j, \ldots$ (resp. Greek letters, $\alpha, \beta, \ldots$) for indices which take values over the range $\{1,\ldots,n\}$ (resp. $\{n+1,\ldots,n+p\}$) and we use the Einstein convention for summation over crossed indices. The Gauss, Codazzi and Ricci equations are given, respectively, by

\begin{equation}\label{eq:G-C-R}
  \left\{
    \begin{aligned}
      \dif \omega^i_j &= - \omega^i_k \wedge \omega^k_j  - \omega^i_\alpha \wedge \omega^\alpha_j + 1 \,  \omega^i \wedge \omega^j;\\
      \dif \omega^i_\alpha &= - \omega^i_k \wedge \omega^k_\alpha  - \omega^i_\beta \wedge \omega^\beta_\alpha;\\
      \dif \omega^\alpha_\beta &=- \omega^\alpha_i \wedge \omega^i_\beta - \omega^\alpha_\gamma \wedge \omega^\gamma_\beta.
    \end{aligned}
  \right.
\end{equation}
So, the curvature form $\Omega^i_j$ and normal curvature form $\Omega^\alpha_\beta$ are 
\begin{equation}\label{eq:Curvature}
  \left\{
    \begin{aligned}
  \Omega^i_j &= - \omega^i_\alpha \wedge \omega^\alpha_j + 1 \, \omega^i \wedge \omega^j,\quad &\Omega^i_j=\frac{1}{2}R^i_{jkl}\,\omega^k \wedge \omega^l;\\
  \Omega^\alpha_\beta &=- \omega^\alpha_i \wedge \omega^i_\beta , &\Omega^\alpha_\beta=\frac{1}{2}R^\alpha_{\beta kl}\,\omega^k \wedge \omega^l.
    \end{aligned}
  \right.
\end{equation}
We restrict these forms to $M^n$. Then 
\begin{equation*}
\omega^\alpha=0. 
\end{equation*}
Since 
\begin{equation*}
0=\dif \omega^\alpha = -\omega^\alpha_i \wedge \omega^i, 
\end{equation*}
by Catan's lemma we may write
\begin{equation}\label{eq:Cartan}
      \omega^\alpha_i =h^\alpha_{ij}\, \omega^j ,\quad h^\alpha_{ij}=h^\alpha_{ji}.	
\end{equation}
We call 
\begin{equation}\label{eq:2forma}
     \B=h^\alpha_{ij}\, \omega^i \otimes \omega^j \otimes \e\alpha
\end{equation}
the {\it second fundamental form} of the immersed manifold $M^n.$ The  {\it mean curvature vector} of $M^n$ is given by
\begin{equation}\label{eq:H}
     \HH=\frac{1}{n}\sum_i h^\alpha_{ii}\, \e\alpha.
\end{equation}
An immersion is said to be {\it minimal} if its mean curvature vanishes identically, i.e., if $\sum_i h^\alpha_{ii}=0$ for all $\alpha.$

\subsection*{Hypersurfaces}In the special case of hypersurfaces in $\mathbb{S}^5$, ($n=4, p=1$), taking a local orthonormal frame such that 
\begin{equation*}
h^5_{ij}=\lambda_i\, \delta^i_j,
\end{equation*}
then the second fundamental form writes
\begin{equation*}\label{eq:2formaDiag}
     \B=\lambda_{i}\, \omega^i \otimes \omega^i \otimes \e5,
\end{equation*}
the equation \eqref{eq:Cartan} is written 
\begin{equation}\label{eq:Cartan_Hyper}
      \omega^5_i =\lambda_{i}\, \omega^i.	
\end{equation}
the equations \eqref{eq:structure} writes
\begin{equation}\label{eq:structureHiper}
  \left\{
    \begin{aligned}
      \dif \mathbf{g}   &= \omega^i \otimes \e{i} ;\\
      \dif \e{i} &= - \omega^i \otimes \mathbf{g} +\omega^j_i\otimes \e{j}+\lambda_i\, \omega^i\otimes \e{5} ;\\
      \dif \e{5} &= -\lambda_i\, \omega^i\otimes \e{i};\\
      \dif \omega^i &= - \omega^i_j \wedge \omega^j ,\quad 	\omega^i_j+\omega^j_i=0 ;\\
      \dif \omega^5 &= 0 ,
    \end{aligned}
  \right.
\end{equation}
and the Gauss and Codazzi equations  \eqref{eq:G-C-R} become
\begin{equation}\label{eq:G-C_Hyper}
  \left\{
    \begin{aligned}
      \dif \omega^i_j &= - \omega^i_k \wedge \omega^k_j  + (1+\lambda_i\lambda_j) \,  \omega^i \wedge \omega^j;\\
      \dif \omega^5_i &= - \lambda_k\, \omega^i_k \wedge \omega^k.
    \end{aligned}
  \right.
\end{equation}

The \emph{$r$-mean curvatures} $H_r$ of an immersion $\mathbf{g}\!:M^4 \hookrightarrow \mathbb{S}^{5}$, with principal curvatures  $\lambda_{1}$, $\lambda_{2}$, $\lambda_{3}$ and $\lambda_{4}$, are given by
\begin{equation*}
H_r=\frac{1}{\binom{n}{k}} \sum_{i_1 <i_2 \cdots < i_r} \lambda_{i_1}\lambda_{i_2} \cdots \lambda_{i_r}, 
\end{equation*}
viz., 
\begin{equation}\label{eq:Hr}
  \left\{
    \begin{aligned}
      H_1 &= \frac14 (\lambda_1 + \lambda_2 + \lambda_3 + \lambda_4);\\
      H_2 &= \frac16 (\lambda_1\lambda_2 + \lambda_1\lambda_3 + \lambda_1\lambda_4 + \lambda_2\lambda_3 + \lambda_2\lambda_4 + \lambda_3\lambda_4);\\
      H_3 &= \frac14 (\lambda_1\lambda_2\lambda_3 + \lambda_1\lambda_2\lambda_4 + \lambda_1\lambda_3\lambda_4 + \lambda_2\lambda_3\lambda_4);\\
      H_4 &= \lambda_1\lambda_2\lambda_3\lambda_4.
    \end{aligned}
  \right.
\end{equation}

Note that $H_1$ and $H_4$ are, respectively, the \emph{mean curvature} and the \emph{Gauss-Kronecker curvature} of the hypersurface.

\subsection*{Immersed surfaces}
In the case of immersed surfaces in $\mathbb{S}^5$ ($n=2$, $p=3$), equations (\ref{eq:Curvature}) and (\ref{eq:Cartan}) yield that the Gaussian curvature of $M^2$ is given  by
\begin{equation}\label{eq:Gauss_Curvature}
\K= \Omega^1_2(\e1,\e2) = \dif \omega^1_2(\e1,\e2)=1+\sum_{\alpha=3}^5
\left |
\begin{array}{cc}
h^\alpha_{11} & h^\alpha_{12}\\
h^\alpha_{21} & h^\alpha_{22}
\end{array}
\right |.
\end{equation}
The \emph{scalar normal curvature} of the immersion $\mathbf{g}$,  $\KN$,  is the length of the normal curvature form,
\begin{equation}\label{eq:Normal_Curvature}
\KN=\sum_{i,j,\alpha,\beta} \left(R^{\alpha}_{\beta i j}\right)^2,\quad R^{\alpha}_{\beta i j}=\sum_k
\left |
\begin{array}{cc}
h^\alpha_{ki} & h^\alpha_{kj}\\
h^\beta_{ki} & h^\beta_{kj}
\end{array}
\right |.
\end{equation}  
Let $\SF$ be the square of the length of the second 
fundamental form,
\begin{equation*}
\SF=\sum \left( h^{\alpha}_{ij}\right)^2.
\end{equation*}
Note that if  $\mathbf{g}$ is minimal, then
%\begin{equation}\label{eq:KG_KN}
%\begin{aligned}
\begin{equation*}
\K=1-\sum_{\alpha}\left[ \left( h^{\alpha}_{11}\right)^2+\left( h^{\alpha}_{12}\right)^2\right], \quad \KN=8\sum \left(h^{\alpha}_{11} h^{\beta}_{12}-h^{\alpha}_{12}h^{\beta}_{11}\right)^2
\end{equation*}
i.e,
\begin{equation}\label{eq:KG_KN}
K=1-\frac{1}{2}\SF , \quad  \KN=4\left[\left(R^{3}_{412}\right)^2+\left(R^{3}_{512}\right)^2 +\left(R^{4}_{512}\right)^2\right].
%\end{aligned}
\end{equation}

 The {\it curvature ellipse} $\mathcal{E}_p$ of  $\mathbf{g}$ at $p$ is the image of the unitary circle by the second fundamental form $\B$ of  $\mathbf{g}$ at $p:$
 
\begin{equation*}
\mathcal{E}_p=\{\B_p(X,X) \in (T_pM)^{\perp} : X \in T_p M,\ \|X\|=1\}.
\end{equation*}
For $X=\cos \theta \, \e1 +\sin \theta \,\e2$, it is easy to see that

\begin{equation}\label{eq:elipse_curvatura}
\B_p(X,X)=\HH_p + 
\left[
\begin{array}{ccc}
\e3 & \e4 & \e5
\end{array}
\right]
\cdot
\left[
\begin{array}{cc}
\frac{h^3_{11}-h^3_{22}}{2} & h^3_{12}\\
\frac{h^4_{11}-h^4_{22}}{2} & h^4_{12}\\
\frac{h^5_{11}-h^5_{22}}{2} & h^5_{12}
\end{array}
\right]
\cdot
\left[
\begin{array}{cc}
\cos 2\theta\\
\sin 2\theta
\end{array}
\right].
\end{equation} 
If  $\mathbf{g}$ is minimal, then
\begin{equation}\label{eq:matrix_ellipse}
\B_p(X,X)=
\left[
\begin{array}{ccc}
\e3 & \e4 & \e5
\end{array}
\right]
\cdot
\left[
\begin{array}{cc}
h^3_{11} & h^3_{12}\\
h^4_{11} & h^4_{12}\\
h^5_{11} & h^5_{12}
\end{array}
\right]
\cdot
\left[
\begin{array}{cc}
\cos 2\theta\\
\sin 2\theta
\end{array}
\right].
\end{equation}
Observe that saying $\KN \neq 0$ is equivalent to say that the rank of the matrix $\left( h_{1j}^{\alpha} \right)_{3 \times 2}$ in \eqref{eq:matrix_ellipse} is two and $\B_p(T_pM)$ is two-dimensional.  

The map  $\mathbf{g}$ is {\it superminimal} if in addition, the curvature ellipse is always a circle. That is equivalent to 
\begin{equation}\label{eq:circulo_curvatura}
\|\B_{11}\|=\|\B_{12}\| ,\quad \langle \B_{11},\B_{12}\rangle=0, \ \ \B_{ij}:=\B(\e{i}, \e{j}).
\end{equation} 

\section{Example of a minimal hypersurface with vanishing Gauss-Kronecker curvature}

In this section, we shall give an example of a minimal hypersurface in $\mathbb{S}^5$ with vanishing Gauss-Kronecker  and 3-mean curvatures. This example is a 2-spherical local bundle over a minimal surface in $\mathbb{S}^5$.

Before to construct the example, we give a result concerning an immersed superminimal surface with vanishing scalar normal curvature.   

\begin{proposition}\label{Proposition1}
Let $M^2$ be a complete surface and $\mathbf{g}\! : M^2\hookrightarrow \mathbb{S}^5$ be a superminimal immersion. If the scalar normal curvature $\KN$ vanishes, then $M^2$ is compact and $\mathbf{g}(M^2)$ is a totally geodesic sphere in $\mathbb{S}^5$. 
\end{proposition}  
\begin{proof}
It is easy to see that
\begin{equation*}\label{eq:superminima}
\begin{aligned}
(\K-1)^2-\frac14 \KN&=\left(\|\B_{11}\|^2+\|\B_{12}\|^2\right)^2-
4\sum_{\alpha<\beta}
\left |
\begin{array}{cc}
h^\alpha_{11} & h^\alpha_{12}\\
h^\beta_{11} & h^\beta_{12}
\end{array}
\right |
\\
&=\left(\|\B_{11}\|^2+\|\B_{12}\|^2\right)^2-4\left( \|\B_{11}\|^2\|\B_{12}\|^2- \langle \B_{11} , \B_{12}\rangle^2\right)\\
&=\left(\|\B_{11}\|^2-\|\B_{12}\|^2\right)^2+4\langle \B_{11} , \B_{12}\rangle^2.
\end{aligned}
\end{equation*}
Therefore, it follows from (\ref{eq:circulo_curvatura}) that  $\mathbf{g}$ is superminimal if and only if 
\begin{equation}
(\K-1)^2-\frac{1}{4}\KN=0.
\end{equation}

Assume now that $\mathbf{g}\! : M^2\hookrightarrow \mathbb{S}^5$ is superminimal and that $\KN=0.$ So, the vectors $\B_{11}$ and $\B_{12}$ are orthogonal and linearly dependent, thus $\B \equiv 0$ and $\K=1$. Therefore, if $M^2$ is complete, by Bonnet-Myers' Theorem, $M^2$ is compact and, since $\B \equiv 0$,    $\mathbf{g}(M^2)$ is a totally geodesic sphere in $\mathbb{S}^5$.
\end{proof}

Now, starting the construction of  the example, let $\mathbf{g}\! : M^2\hookrightarrow \mathbb{S}^5$ be an isometric immersion and  consider  $\mathcal{N}\subset(\mathbf{g}_*(TM))^\bot$  the unit normal bundle of the immersion  $\mathbf{g}$. Then
\begin{equation*}
\mathcal{N}=\{(p,V) \in M^2 \times \mathbb{R}^6 : \|V\|=1,\ V \perp \mathbb{R}\cdot \mathbf{g}(p) \oplus \mathbf{g}_*(T_pM)\}.
\end{equation*}
Denote the projection to the first factor by $\pi_1\!:\mathcal{N} \rightarrow M^2$ and the projection to the second factor by the map $\mathbf{x}_\mathbf{g}:\mathcal{N} \rightarrow \mathbb{S}^5$. 

Consider 
\begin{equation}\label{eq:N*}
\mathcal{N}_*(p)=\mathcal{N}(p) \setminus \left \{ \B_p(T_pM)^{\perp} \cap \mathbf{g}_*(T_pM)^{\perp} \right\}.
\end{equation}
We have three situations :
\begin{enumerate}
	\item[(a)] if $\KN \neq 0,$ i.e., if rank of the matrix $\left( h_{1j}^{\alpha} \right)_{3 \times 2}$ is two, then $\mathcal{N}_*(p)$ is a $2$-sphere without two antipodal points, given by the orthogonal complement of  $\B_p(T_pM)$ in $\mathbf{g}_*(T_pM)^{\perp};$
	\item[(b)] if rank of the matrix $\left( h_{1j}^{\alpha} \right)_{3 \times 2}$ is one, then $\mathcal{N}_*(p)$ is a $2$-sphere without a great circle;
	\item[(c)] if rank of the matrix $\left( h_{1j}^{\alpha} \right)_{3 \times 2}$ is zero, that is, if $\mathbf{g}(M)$ is totally geodesic, then $\mathcal{N}_*(p)=\emptyset.$ 
\end{enumerate}

\begin{figure}[ht]
	\centering
	\parbox[][2.5cm][t]{0.12\textwidth}{\includegraphics[width=0.12\textwidth]{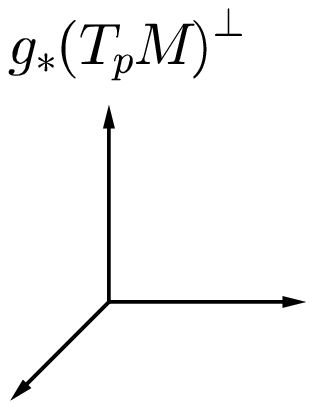}}\ \ 
\subfigure[$\rank(h^\alpha_{1j})=2$]{\parbox[c][2.5cm][t]{0.25\textwidth}{\includegraphics[width=0.25\textwidth]{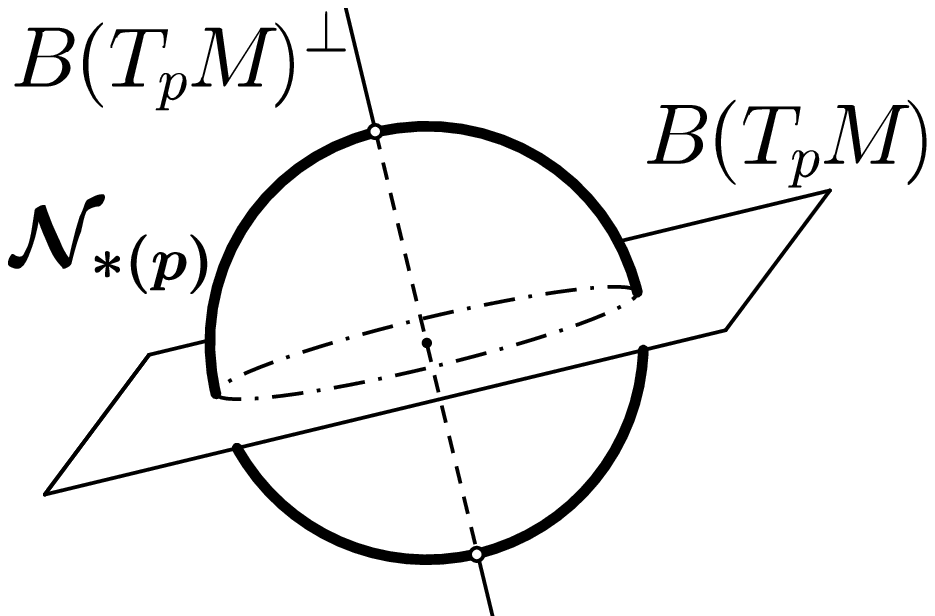}}}\ \ \ \ \ 
\subfigure[$\rank(h^\alpha_{1j})=1$]{\parbox[c][2.5cm][t]{0.24\textwidth}{\includegraphics[width=0.24\textwidth]{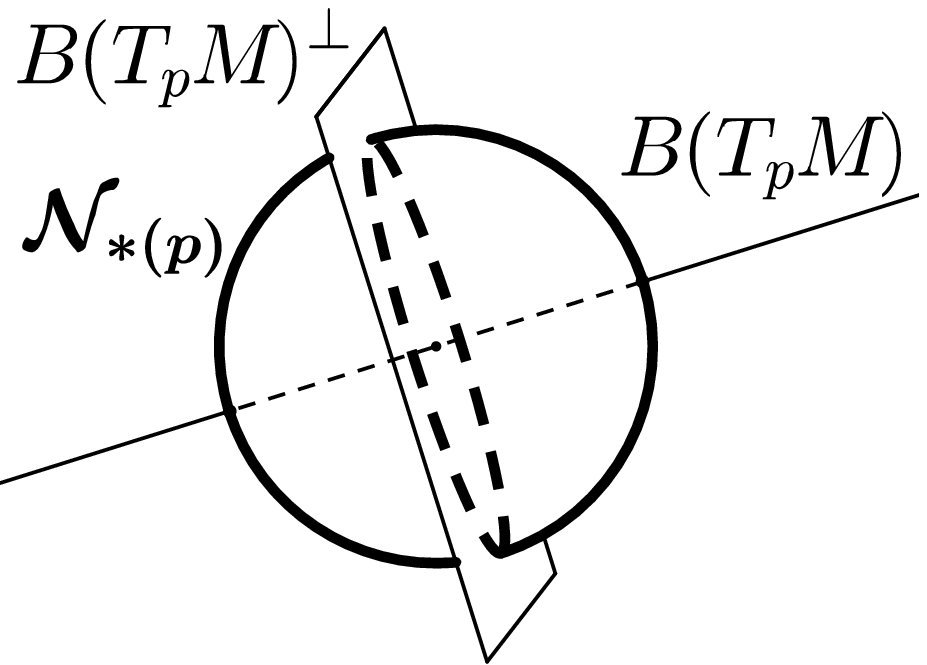}}}\ \ \ \ \  
\subfigure[$\rank(h^\alpha_{1j})=0$]{\parbox[c][2.5cm][t]{0.24\textwidth}{\includegraphics[width=0.24\textwidth]{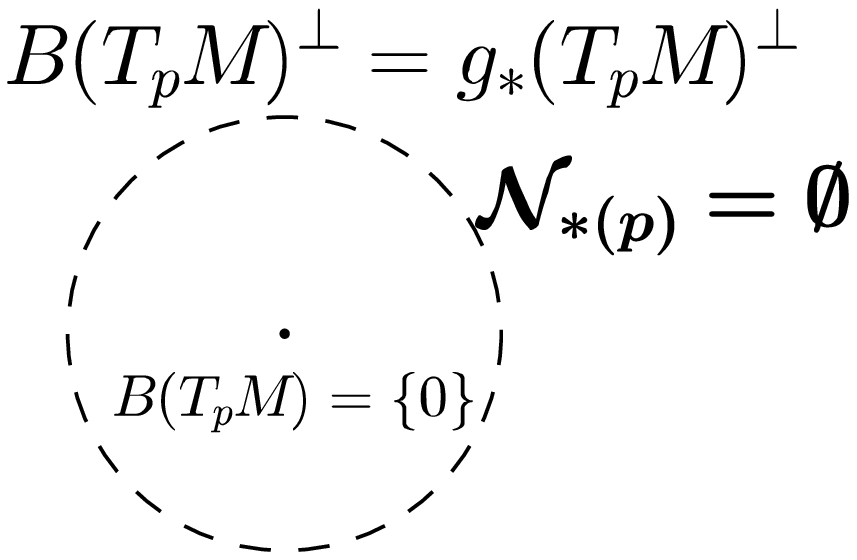}}}
\caption{The three cases of $\mathcal{N}_*(p)$ in terms of matrix $\left( h_{1j}^{\alpha} \right)_{3 \times 2}$}
\label{fig:casos}
\end{figure}

\begin{theorem}\label{Teorema1}
Let $\mathbf{g}\! : M^2\hookrightarrow \mathbb{S}^5$ be a minimal immersion with  nowhere zero second fundamental form. Then there exists an open set $\mathcal{N}_*$ of $\mathcal{N}$ such that $\mathbf{x}_\mathbf{g}:\mathcal{N}_* \rightarrow \mathbb{S}^5$ is an immersion of the $4$-dimensional manifold $\mathcal{N}_*$. Furthermore, $\mathbf{x}_\mathbf{g}:\mathcal{N}_* \rightarrow \mathbb{S}^5$ is an immersed minimal hypersurface of $\mathbb{S}^5$ with zero Gauss-Kronecker curvature and zero $3$-mean curvature.
\end{theorem}
\begin{proof}
Let $\{\e{1}, \e{2}, \e{3}, \e{4}, \e{5}\}$ be a frame defined on an open set $U \subset M^2$ adapted to the immersion  $\mathbf{g}$ and let $W$ be a coordinate neighborhood of $\mathbb{S}^2$. Parameterize $\pi_1^{-1}(U)$ locally by $U \times W$ via the map $\mathbf{y}\!:U \times W \rightarrow \mathcal{N}$ given by
\begin{equation*}
(p, \theta, \varphi)\mapsto(p, \sin \varphi \cos \theta \, \e{3} + \sin \varphi \sin \theta \, \e{4} + \cos \varphi \, \e{5}), \ 0\leq\theta<2\pi, \ 0< \varphi < \pi. 
\end{equation*}
Then $\mathbf{x}\!: U \times W \rightarrow \mathbb{S}^5$ given by 
\begin{equation}\label{xg}
\mathbf{x}(p,\theta,\varphi)=\sin \varphi \cos \theta \, \e{3} + \sin \varphi \sin \theta \, \e{4} + \cos \varphi \, \e{5}
\end{equation}
is a local representation of $\mathbf{x_g}.$ Thus,
\begin{equation*}
\begin{aligned}
\dif \mathbf{x} &=\sin \varphi \cos \theta \, \dif \e{3} + \sin \varphi \sin \theta \, \dif \e{4} + \cos \varphi \, \dif \e{5}+\\
&+\cos \varphi \cos \theta \, \dif \varphi \, \e{3} + \cos \varphi \sin \theta \, \dif \varphi \, \e{4} - \sin \varphi \, \dif \varphi \, \e{5}-\\ 
&-\sin \varphi \sin \theta \, \dif \theta \, \e{3}  + \sin \varphi \cos \theta \, \dif \theta \, \e{4}
\end{aligned}
\end{equation*}
The structure equations adapted to  $\mathbf{g}$ then yield
\begin{equation}\label{eq:dif_x}
\begin{aligned}
\dif \mathbf{x} &=\left(\sin \varphi \cos \theta \ \omega^1_3 + \sin \varphi \sin \theta \ \omega^1_4 + \cos \varphi \ \omega^1_5\right)\otimes \e{1}+\\
&+\left(\sin \varphi \cos \theta \ \omega^2_3 + \sin \varphi \sin \theta \ \omega^2_4 + \cos \varphi \ \omega^2_5\right)\otimes \e{2}+\\
&+\left(-\sin \varphi \sin \theta \ \dif \theta + \cos \varphi \cos \theta \ \dif \varphi + \cos \varphi \ \omega^3_5-\sin \varphi \sin \theta \ \omega^4_3\right)\otimes \e{3}+\\
&+\left(\sin \varphi \cos \theta \ \dif \theta + \cos \varphi \sin \theta \ \dif \varphi + \cos \varphi \ \omega^4_5+\sin \varphi \cos \theta \ \omega^4_3\right)\otimes \e{4}+\\
&+\left(-\sin \varphi \ \dif \varphi + \sin \varphi \cos \theta \ \omega^5_3+\sin \varphi \sin \theta \ \omega^5_4\right)\otimes \e{5}.
\end{aligned}
\end{equation}
From \eqref{eq:Cartan} and \eqref{eq:dif_x}, we obtain the first fundamental form of $\mathbf{x}$ 
\begin{equation*}\label{eq:1form_fund}
\begin{aligned}
\dif s^2_{\mathbf{x}}=
\langle \dif \mathbf{x}, \dif \mathbf{x}\rangle &=
\left[
\begin{array}{cc}
\omega^1 & \omega^2
\end{array}
\right]
\otimes
\left[
\begin{array}{cc}
c_{11} & c_{12}\\
c_{21} & c_{22}
\end{array}
\right]
\cdot
\left[
\begin{array}{cc}
\omega^1\\
\omega^2
\end{array}
\right]+\\
&+\left(-\sin \varphi \sin \theta \ \dif \theta + \cos \varphi \cos \theta \ \dif \varphi + \cos \varphi \ \omega^3_5-\sin \varphi \sin \theta \ \omega^4_3\right)^2+\\
&+\left(\sin \varphi \cos \theta \ \dif \theta + \cos \varphi \sin \theta \ \dif \varphi + \cos \varphi \ \omega^4_5+\sin \varphi \cos \theta \ \omega^4_3\right)^2+\\
&+\left(-\sin \varphi \ \dif \varphi + \sin \varphi \cos \theta \ \omega^5_3+\sin \varphi \sin \theta \ \omega^5_4\right)^2,
\end{aligned}
\end{equation*}
where 
\begin{equation*}
c_{ij}=\sum_k a_{ik} a_{kj}, \quad a_{ij}=\sin \varphi \cos \theta \, h^3_{ij}  + \sin \varphi \sin \theta \, h^4_{ij} + \cos \varphi \, h^5_{ij}.
\end{equation*}
Since $\mathbf{g}$ is minimal then 
\begin{equation*}
C=A^2=\left[
\begin{array}{cc}
a_{11}^2+a_{12}^2 & 0\\
0 & a_{11}^2+a_{12}^2
\end{array}
\right] = 
\left[
\begin{array}{cc}
-\det A & 0\\
0 & -\det A
\end{array}
\right] .
\end{equation*}
Thus
\begin{equation}\label{eq:1form_fund2}
\begin{aligned}
\dif s^2_{\mathbf{x}}&=
-\det A \left( (\omega^1)^2+(\omega^2)^2\right)+\\
&+\left(-\sin \varphi \sin \theta \ \dif \theta + \cos \varphi \cos \theta \ \dif \varphi + \cos \varphi \ \omega^3_5-\sin \varphi \sin \theta \ \omega^4_3\right)^2+\\
&+\left(\sin \varphi \cos \theta \ \dif \theta + \cos \varphi \sin \theta \ \dif \varphi + \cos \varphi \ \omega^4_5+\sin \varphi \cos \theta \ \omega^4_3\right)^2+\\
&+\left(-\sin \varphi \ \dif \varphi + \sin \varphi \cos \theta \ \omega^5_3+\sin \varphi \sin \theta \ \omega^5_4\right)^2,
\end{aligned}
\end{equation}
the quadratic form $-\det A \left( (\omega^1)^2+(\omega^2)^2\right)$ is positive and is positive-definite if and only if $\det C\neq0$. Since
\begin{equation*}
\begin{aligned}
\det C &= \left[\left(\sin \varphi \cos \theta \, h^3_{11}  + \sin \varphi \sin \theta \, h^4_{11} + \cos \varphi \, h^5_{11}\right)^2\right.+\\
&+\left.\left(\sin \varphi \cos \theta \, h^3_{12}  + \sin \varphi \sin \theta \, h^4_{12} + \cos \varphi \, h^5_{12}\right)^2\right]^2,
\end{aligned}
\end{equation*}
we get

\begin{equation}\label{eq:system1}
\det C = 0 \quad \Longleftrightarrow \quad 
\left[
\begin{array}{ccc}
h^3_{11} & h^4_{11} & h^5_{11}\\
h^3_{12} & h^4_{12} & h^5_{12}
\end{array}
\right]
\cdot
\left[
\begin{array}{c}
\sin \varphi \cos \theta\\
\sin \varphi \sin \theta\\
\cos \varphi
\end{array}
\right]
=
\left[
\begin{array}{c}
0\\
0
\end{array}
\right].
\end{equation}

From the hypothesis that the second fundamental form is nonzero, there are two possibilities at each point $p$\,: the rank of $\left( h_{1j}^{\alpha} \right)_{3 \times 2}$ is either two or one. We claim that in both cases $\det C\neq 0$ in all point $(p,\theta,\varphi)\in \mathcal{N}_*(p)$. 

\begin{itemize}
	\item[(i)] Suppose ${\rank(\left( h_{1j}^{\alpha} \right)_{3 \times 2})=2}$. In this case, we may choose a particular orthonormal local frame $\{\e1, \e2,\e3, \e4, \e5\}$ such that the matrix $\left( h_{1j}^{\alpha} \right)_{3 \times 2}$ takes the form
 \begin{equation*}
	\left[
	\begin{array}{cc}
		\alpha &\beta\\
		0&\gamma\\
		0&0
	\end{array}
	\right],
\end{equation*}
with $\alpha\gamma\neq0$. In fact, starting from $\{\e1, \e2, \e3, \e4, \e5\}$, an orthonormal frame,  with generic matrix $\left( h_{1j}^{\alpha} \right)$, let 
\begin{align*}
	{\te3}&=h^3_{11}\e3+h^4_{11}\e4+h^5_{11}\e5\\
	{\te4}&=h^3_{12}\e3+h^4_{12}\e4+h^5_{12}\e5.
\end{align*}
One may apply the Gram Schmidt process on these two vectors to obtain the vectors $\{\tte3, \tte4\}$. Then we chose $\te5$, such that $\{\e1, \e2, \tte3, \tte4, \te5\}$ be orthonormal. This is the  required basis. It is easy to verify that $\alpha = \| \te3\|$, $\alpha\beta = \langle \te3, \te4\rangle$ and $\beta^2+\gamma^2 = \| \te4\|^2$.  %Note that the vectors $\tte3$ and $\tte4$ give the directions of the axis of the curvature ellipse.    

For a such particular frame, the equivalence \eqref{eq:system1} reduces to
\begin{equation*}
	\det C = 0 \Longleftrightarrow \left\{
	\begin{array}{rc}
	\alpha \sin\varphi\cos \theta &= 0\\
	\beta \sin\varphi\cos \theta + \gamma \sin\varphi\sin \theta &= 0
	\end{array}
	\right. .
\end{equation*}
Since $\alpha\gamma\neq0$ and $0<\varphi<\pi$, then $\det C\neq0$ for all point $(p, \theta, \varphi)\in \mathcal{N}_*(p)$, given in \eqref{eq:N*}.
 
\item[(ii)] Suppose ${\rank(\left( h_{1j}^{\alpha} \right)_{3 \times 2})=1}$. We may choose a particular orthonormal local frame $\{\e1, \e2,\e3, \e4, \e5\}$ such that the matrix $\left( h_{1j}^{\alpha}(p) \right)_{3 \times 2}$, at point $p$, takes the form
 \begin{equation*}
	\left[
	\begin{array}{cc}
		\alpha &\beta\\
		0&0\\
		0&0
	\end{array}
	\right],
\end{equation*}
with $\alpha^2+\beta^2\neq0$. In this case, the system \eqref{eq:system1} takes the form
\begin{equation*}
	\det C = 0 \Longleftrightarrow \left\{
	\begin{array}{rc}
	\alpha \sin\varphi\cos \theta &= 0\\
	\beta \sin\varphi\cos \theta &= 0
	\end{array}
	\right. .
\end{equation*}
Since $\alpha^2+\beta^2\neq0$  and $0<\varphi<\pi$, then $\det C\neq0$ if and only if $\theta$ is different from $\frac\pi 2$ and $\frac{3\pi}{2}$. Then $\det C\neq0$  for all point $(p, \theta, \varphi)\in \mathcal{N}_*(p)$, given in \eqref{eq:N*}.
\end{itemize}

Let 
\begin{equation*}
\mathbf{w}= a \, \e{1} + b \, \e{2} + c \, \frac{\partial}{\partial \theta}+ d\, \frac{\partial}{\partial \varphi}, \quad a, \, b, \, c, \, d \, \in \mathbb{R}
\end{equation*}
be a vector in $T_pU \times T_p W$. The 1-forms $\omega^{\alpha}_{\beta}$, are the connection forms of the normal bundle of $\mathbf{g}$ and $\omega^{\alpha}_{\beta}(\frac{\partial}{\partial \theta})$,  $\omega^{\alpha}_{\beta}(\frac{\partial}{\partial \varphi})$ are zero. Suppose now 
\begin{equation*}
\dif \mathbf{x}_{(p,\theta,\varphi)} (\mathbf{w})=0.
\end{equation*}
From equation \eqref{eq:1form_fund2} and from the facts that $\dif s^2_{\mathbf{x}}(\mathbf{w})=0$ and $\det C\neq0$, we get $a=b=0$ and $c$ and $d$ must satisfy 
\begin{equation*}\label{eq:system2}
\left[
\begin{array}{cc}
-\sin \varphi \, \sin \theta  & \cos \varphi \, \cos \theta\\
\sin \varphi \, \cos \theta  & \cos \varphi \, \sin \theta\\
0 & -\sin \varphi
\end{array}
\right]
\left[
\begin{array}{c}
c\\
d
\end{array}
\right]
=
\left[
\begin{array}{c}
0\\
0\\
0
\end{array}
\right].
\end{equation*}
Since $0 < \varphi < \pi$, the first matrix on the equation above has rank 2 and necessarily $c=d=0$. So,
\begin{equation*}
\dif \mathbf{x}_{(p,\theta,\varphi)} (\mathbf{w})=0 \quad  \Rightarrow \ \mathbf{w}=0.
\end{equation*} 
Therefore, $\mathbf{x_g}$ is an immersion.

Let $\left\{\e{1}, \e{2}, \frac{\partial}{\partial \theta}, \frac{\partial}{\partial \varphi}, \mathbf{g} \circ \pi_1 \right\}$ an orthogonal frame adapted to the immersion $\mathbf{x_g}$. Note that $\mathbf{g} \circ \pi_1(p, \theta, \varphi)$ is normal to $T_{\mathbf{x}(p, \theta, \varphi)}\mathcal{N_*}$ and tangent to $\mathbb{S}^5.$ The second fundamental form of $\mathbf{x_g}$ is given by
\begin{equation*}
\begin{aligned}
\mathrm{II}_{\mathbf{x_g}}&=-\langle \dif \mathbf{x}, \dif(\mathbf{g} \circ \pi_1) \rangle\\
&=-\langle \dif \mathbf{x}, \omega^1 \otimes \e{1}+\omega^2 \otimes \e{2} \rangle\\
&=a_{11} \, \omega^1 \otimes \omega^1+2\, a_{12} \, \omega^1 \otimes \omega^2+a_{22} \, \omega^2 \otimes \omega^2.
\end{aligned}
\end{equation*}
Since $\mathbf{g}$ is minimal, then the trace of matrix $A=(a_{ij})$ is zero. Thus, trace of matrix of the $\mathrm{II}_{\mathbf{x_g}}$ is zero too. Therefore $\mathbf{x}_\mathbf{g}\!:\mathcal{N}_* \rightarrow \mathbb{S}^5$ is a minimal immersion. Since the rank of matrix $A$ is two, then $3$ and $4$-mean curvatures vanish. This proves the theorem.  
\end{proof}

%%%%%%%%%%%%%%%%%%%%%%%%%%%%%%%%%%

\section{Minimal hypersurfaces in $\mathbb{S}^5$ with zero Gauss-Kronecker and $3$-mean curvatures}

In this section we shall study locally the complete minimal hypersurfaces in $\mathbb{S}^5$ with zero Gauss-Kronecker curvature, zero $3$-mean curvature and nowhere zero second fundamental form.  First we need the following lemmas.

\begin{lemma}\label{lemma1}
Let $\mathbf{g}\! :M^4 \rightarrow \mathbb{S}^5$ be an immersed hypersurface of $\mathbb{S}^5$ and let \linebreak $\{\e1,\e2,\e3,\e{4}, \e{5}\}$  be an orthonormal (local) frame adapted to $\mathbf{g}$ with dual co-frame  $\{\omega^1, \omega^2,\omega^3, \omega^4, \omega^5\}$. If $\{ \e1, \e2, {\te{3}}, {\te{4}}, \e{5}\}$ is other orthonormal frame with dual co-frame  $\{\omega^1, \omega^2,\widetilde{\omega}^3, \widetilde{\omega}^4, \omega^5\}$ such that 
\begin{equation*}
\left[
\begin{array}{c}
{\te{3}}\\
{\te{4}}
\end{array}
\right]
=
\left[
\begin{array}{cr}
\cos \theta & -\sin \theta\\
\sin \theta & \cos \theta
\end{array}
\right]
%\cdot
\left[
\begin{array}{cc}
\e{3}\\
\e{4}
\end{array}
\right].
\end{equation*}
Then 
\begin{equation*}
\left[
\begin{array}{c}
\widetilde{\omega}^3\\
\widetilde{\omega}^4
\end{array}
\right]
=
\left[
\begin{array}{cr}
\cos \theta & -\sin \theta\\
\sin \theta & \cos \theta
\end{array}
\right]
%\cdot
\left[
\begin{array}{cc}
\omega^3\\
\omega^4
\end{array}
\right];
\end{equation*}

\begin{equation*}
\left[
\begin{array}{cc}
\widetilde{\omega}^1_3 & \widetilde{\omega}^2_3\\
\widetilde{\omega}^1_4 & \widetilde{\omega}^2_4
\end{array}
\right]
=
\left[
\begin{array}{cr}
\cos \theta & -\sin \theta\\
\sin \theta & \cos \theta
\end{array}
\right]
%\cdot
\left[
\begin{array}{cc}
\omega^1_3 & \omega^2_3\\
\omega^1_4 & \omega^2_4
\end{array}
\right];
\end{equation*}
\begin{equation*}
\widetilde{\omega}^1_2=\omega^1_2, \quad \widetilde{\omega}^3_4=\omega^3_4+\dif \theta, \quad \widetilde{\omega}^5_i=\omega^5_i.
\end{equation*}

%\begin{equation*}
%\widetilde{\omega^1_2}=\omega^1_2, 
%\end{equation*}
\end{lemma}
\begin{proof}
Taking exterior differentiate of 
\begin{equation*}
\left\{
\begin{aligned}
\omega^3&=\cos \theta \, \widetilde{\omega}^3 + \sin \theta \, \widetilde{\omega}^4,\\ 
\omega^4&=-\sin \theta \, \widetilde{\omega}^3 + \cos \theta \, \widetilde{\omega}^4
\end{aligned}
\right.
\end{equation*}
and  using \eqref{eq:structureHiper} we obtain
\begin{equation*}
\begin{aligned}
\dif \widetilde{\omega}^3&=-\left( \cos \theta \, \omega^3_1 - \sin \theta \, \omega^4_1 \right) \wedge \omega^1-\left( \cos \theta \, \omega^3_2 - \sin \theta \, \omega^4_2 \right) \wedge \omega^2-\left( \omega^3_4+\dif \theta\right)\wedge \widetilde{\omega}^4\\
\dif \widetilde{\omega}^4&=-\left( \sin \theta \, \omega^3_1 + \cos \theta \, \omega^4_1 \right) \wedge \omega^1-\left( \sin \theta \, \omega^3_2 + \cos \theta \, \omega^4_2 \right) \wedge \omega^2+\left( \omega^3_4+\dif \theta\right)\wedge \widetilde{\omega}^3.
\end{aligned}
\end{equation*}
\end{proof}

\begin{lemma}\label{lemma2}
In the hypothesis above. If the functions $f_3,  f_4,  g_3$ and $g_4$ are defined by
\begin{equation}\label{eq:defgf}
\left\{
\begin{aligned}
\omega^3_2=f_3 \, \omega^1+g_3 \, \omega^2\\
\omega^4_2=f_4 \, \omega^1+g_4 \, \omega^2.
\end{aligned}
\right.
\end{equation}
Then 
\begin{equation*}
\left[
\begin{array}{cc}
\widetilde{f}_3 & \widetilde{g}_3\\
\widetilde{f}_4 & \widetilde{g}_4
\end{array}
\right]
=
\left[
\begin{array}{cr}
\cos \theta & -\sin \theta\\
\sin \theta & \cos \theta
\end{array}
\right]
\left[
\begin{array}{cc}
f_3 & g_3\\
f_4 & g_4
\end{array}
\right]
\end{equation*}
and
\begin{equation*}
\begin{aligned}
\widetilde{f}_3^2 + \widetilde{f}_4^2&=f_3^2 +f_4^2 \, ,\quad \widetilde{f}_3\widetilde{g}_3 + \widetilde{f}_4\widetilde{g}_4&=f_3g_3+f_4g_4,\\
\widetilde{g}_3^2+\widetilde{g}_4^2&=g_3^2 + g_4^2 \, , \quad \widetilde{f}_3\widetilde{g}_4 - \widetilde{f}_4\widetilde{g}_3&=f_3g_4- f_4g_3.
\end{aligned}
\end{equation*}
\end{lemma}
\begin{proof}
Since
\begin{equation*}
\left\{
\begin{aligned}
\e{3}&=\cos \theta \, {\te{3}} + \sin \theta \, {\te{4}},\\ 
\e{4}&=-\sin \theta \, {\te{3}} + \cos \theta \, {\te{4}}.
\end{aligned}
\right.
\end{equation*}
It follows from \eqref{eq:defgf} that
\begin{equation*}
\begin{aligned}
\left(\nabla \e{2}\right)&=\omega^1_2  \otimes \e{1}+\omega^3_2 \otimes \e{3} + \omega^4_2 \otimes \e{4}=\\
&=\omega^1_2  \otimes \e{1}+\left[\left(\cos \theta \, f_3 - \sin \theta \, f_4 \right) \, \omega^1+\left( \cos \theta \, g_3 - \sin \theta \, g_4 \right)\omega^2 \right]\otimes {\te{3}}+\\
&+\left[\left(\sin \theta \, f_3 + \cos \theta \, f_4 \right) \, \omega^1+\left( \sin \theta \, g_3 + \cos \theta \, g_4 \right)\omega^2 \right]\otimes {\te{4}},
\end{aligned}
\end{equation*}
where $\nabla$ is the Levi-Civita connections of $M^4$.

\end{proof}

\begin{theorem}\label{Teorema2}
Let $\mathbf{g}\! :M^4 \rightarrow \mathbb{S}^5$ be a complete, oriented minimal immersed hypersurface of $\mathbb{S}^5$ with zero Gauss-Kronecker curvature and zero $3$-mean curvature. If the square $S$ of the length of the second fundamental form is nowhere zero, then there exists a minimal immersed $\tilde{\eta}\!: \mathbb{V}^2 \rightarrow \mathbb{S}^5$ and a local isometry $\tau \! : M^4_{*} \to \mathcal{N}_*$, where $M^4_{*}=M^4 \setminus \left \{\mathbf{g}^{-1}\left( \B_{\pi(p)}(T_{\pi(p)}\mathbb{V})^{\perp} \cap \mathbf{\widetilde{\eta}}_*(T_{\pi(p)}\mathbb{V})^{\perp} \right)\right\}$, such that $\eta|_{M^4_*}=\widetilde{\eta}\circ \pi_1 \circ \tau$ and $\mathbf{x_{\boldsymbol{\widetilde{\eta}}}} \circ \tau=\mathbf{g}|_{M^4_{*}}.$  
\end{theorem}

\begin{proof}
Let $\mathcal{B} =\{ \e1, \e2, \e{3}, \e{4}, \e{5}\}$ be an orthonormal frame on $\mathbb{S}^{5}$ \emph{adapted} to the immersion $\mathbf{g}\! :M^4 \rightarrow \mathbb{S}^5$ in the sense that $\{\e1,\ldots,\e{4}\}$ spans $\mathbf{g}_*(TM)$, $\e{5}$ determines a global Gauss map ${\eta} \! : M^4 \rightarrow \mathbb{S}^5$ and
as described in \eqref{eq:Cartan_Hyper} holds
\begin{equation*}
      \omega^5_i =\lambda_{i}\, \omega^i, \quad i=1,2,3,4.	
\end{equation*}
By hypothesis and \eqref{eq:Hr}, we have exactly two of the $\lambda_i$ are zero and the sum of the other two is zero. Hence, we may assume that 
\begin{equation*}
      \lambda=\lambda_1=-\lambda_2>0, \, \lambda_3=\lambda_4=0
\end{equation*} 
and
\begin{equation}\label{eq:lambda}
\omega^5_1=\lambda \, \omega^1, \quad \omega^5_2=-\lambda \, \omega^2, \quad \omega^5_3=\omega^5_4=0.
\end{equation}
In this case, we have two directions and one plane well defined. Since $S>0$ on $M^4$ then 
\begin{equation*}
\mathcal{D}_p:=\left\{v \in T_pM^4 :  \B_p(v,w)=0, \ \forall \ w \in T_pM^4\right\} 
\end{equation*}
is a $2$-dimensional distribution. Now, we proceed to show that $\mathcal{D}_p$ is involutive, that its integral surfaces $\mathcal{F}^2$ are totally geodesic in $M^4$ and $\mathbf{g}(\mathcal{F}^2)$ are totally geodesic in $\mathbb{S}^5.$ 
From \eqref{eq:lambda} and Codazzi equations given in \eqref{eq:G-C_Hyper} we have that 
\begin{equation}\label{eq:C_Hyper}
  \left\{
    \begin{aligned}
      \omega^3_1 \wedge \omega^1-\omega^3_2 \wedge \omega^2=0;\\
       \omega^4_1 \wedge \omega^1-\omega^4_2 \wedge \omega^2=0.
    \end{aligned}
  \right.
\end{equation}
Let $\mathfrak{I}\left(\omega^1, \omega^2\right)$ be the ideal generated by $\omega^1$ and  $\omega^2$. By Cartan's lemma \ $\omega^3_1, \ \omega^3_2, \ \omega^4_1$ and $\omega^4_2$ \ belong to the ideal $\mathfrak{I}\left(\omega^1, \omega^2\right)$. From \eqref{eq:structureHiper} follows  
\begin{equation*}
\dif \omega^i \, \in \, \mathfrak{I}\left(\omega^1, \omega^2\right), \ i=1,2  \quad\mathrm{or} \quad \omega^i\left([\e{3}, \, \e{4}]\right)=0, \ i=1,2.
\end{equation*}
Therefore, $\mathcal{D}_p$ is involutive.  Then by Frobenius Theorem there exists an unique  maximal, connected, integral surface $\mathcal{F}^2$ of $\mathcal{D}_p$ through $p.$ Since $\mathbf{g}\! : M^4 \rightarrow \mathbb{S}^5$ is an immersion of a complete manifold, a theorem of D. Ferus \cite{Ferus.1971} implies that $\mathcal{F}^2$ is complete and totally geodesic in $M^4.$   
From \eqref{eq:C_Hyper} it is easy to see that 
\begin{equation}\label{eq:FC1}
  \left\{
    \begin{aligned}
     & \omega^3_2(\e{1}) + \omega^3_1(\e{2}) =0;\\
      &\omega^4_2(\e{1}) + \omega^4_1(\e{2}) =0;\\
      &\omega^i_j(\e{k})=0, \ i=1,2; \ j,k=3,4. 
    \end{aligned}
  \right.
\end{equation}
Substituting \eqref{eq:lambda} and \eqref{eq:FC1} into \eqref{eq:structureHiper} we obtain 
\begin{equation}\label{eq:FC2}
  \left\{
    \begin{aligned}
      \dif \e{3} (\e{3})&=-\mathbf{g} + \omega^4_3(\e{3})\e{4};\\
      \dif \e{3} (\e{4})&= \ \ \omega^4_3(\e{4})\e{4};\\
      \dif \e{4} (\e{3})&= \ \ \omega^3_4(\e{3})\e{3};\\
      \dif \e{4} (\e{4})&=-\mathbf{g} + \omega^3_4(\e{4})\e{3}.\\ 
    \end{aligned}
  \right.
\end{equation}
 This implies that 
 \begin{equation*}
 \omega^k(\overline{\nabla}_{\e{j}}\e{i})=0, \quad k=1,2; \quad i,j=3,4,
 \end{equation*}
 where $\overline{\nabla}$ is the Levi-Civita connections of $\mathbb{S}^5$. Therefore, $\mathbf{g}(\mathcal{F}^2)$ is totally geodesic in $\mathbb{S}^5.$ Since $\mathcal{F}^2$ is complete and $\left. \mathbf{g}\right|_{\mathcal{F}^2}\! :\mathcal{F}^2 \rightarrow \mathbb{S}^5$ is an isometric immersion, then  $\mathbf{g}(\mathcal{F}^2)$ is an unitary 2-sphere $\mathbb{S}^2$ in $\mathbb{S}^5.$ So, we have that $\left. \mathbf{g}\right|_{\mathcal{F}^2}\! :\mathcal{F}^2 \rightarrow \mathbb{S}^2$ is a covering map (see \cite[p.146, Prop.6.16]{Lima.2003}). Thus, $\mathcal{F}^2$ is a 2-sphere in $M^4$ and $\left. \mathbf{g}\right|_{\mathcal{F}^2}\! :\mathcal{F}^2 \rightarrow \mathbb{S}^2$ is a diffeomorphism (see \cite[p.141]{Lima.2003}). Therefore, the maximal integral surface $\mathcal{F}^2$ is regular (see \cite[p.98]{Tamura.1992}) and a theorem of R. Palais (see \cite{Palais.1957}) implies that the quotient space   
\begin{equation*}
\mathbb{V}^2=\left. M^4\right/ \mathcal{F}^2
\end{equation*}
can be endowed with a structure of a 2-dimensional differential manifold such that $\pi \! : M^4 \rightarrow \mathbb{V}^2$ is a submersion.
The Gauss map $\eta \! : M^4 \rightarrow \mathbb{S}^5$ induces a smooth map $\widetilde{\eta} \! : \mathbb{V}^2 \rightarrow \mathbb{S}^5$ such that $\widetilde{\eta} \circ \pi= \eta,$ 
\begin{equation*}
\xymatrix{
M^4 \ar[d]_{\pi} \ar[dr]^{\eta} &  \\
\mathbb{V}^2 \ar[r]_{\widetilde{\eta}} & \mathbb{S}^5.
}
\end{equation*}
In fact, from \eqref{eq:lambda}, we have that $\overline{\nabla}_{\e{3}}\eta$ and $\overline{\nabla}_{\e{4}}\eta$ vanish. Thus, $\eta$ is constant along the integral surfaces $\mathcal{F}^2$ and $\widetilde{\eta}$ is well defined.
Let $S$ be a smooth transversal surface to the leaf $\mathcal{F}^2$ of $\mathcal{D}_p$ through a point $p  \in  M^4$ such that  $T_pS = \text{span}\{\e{1}|_p, \e{2}|_p\}$ and $T_p\mathcal{F} = \text{span}\{\e{3}|_p, \e{4}|_p\}$. Since $\pi \! : M^4 \rightarrow \mathbb{V}^2$ is a submersion, then 
\begin{equation*}
\text{span} \left \{\dif \pi_p(\e{1}|_p), \dif \pi_p (\e{2}|_p)\right \} =  T_{\pi(p)} \mathbb{V}^2.  
\end{equation*}
The third formula given in \eqref{eq:structureHiper} implies that  
\begin{equation*}
    \begin{aligned}
   \dif \widetilde{\eta}_{\pi(p)} \left( \dif \pi_p(\e{1}|_p)\right)&=\dif (\widetilde{\eta} \circ {\pi})_p \left(\e{1}|_p\right)=\dif \eta_p \left(\e{1}|_p\right)=-\lambda(p) \, \e{1}|_p;\\
   \dif \widetilde{\eta}_{\pi(p)} \left( \dif \pi_p(\e{2}|_p)\right)&=\lambda(p) \, \e{2}|_p.
    \end{aligned}
\end{equation*}
So, the first fundamental form of $\widetilde{\eta}$ is given by
\begin{equation*}
ds^2_{\widetilde{\eta}}= \lambda^2 \left( \omega^1 \otimes \omega^1 + \omega^2 \otimes \omega^2\right).
\end{equation*}
Since $\lambda(p)>0$ for all $p \in M^4$, then $\widetilde{\eta} \! : \mathbb{V}^2 \rightarrow \mathbb{S}^5$ is an immersion.  Now, we will show that $\widetilde{\eta}$ is a minimal immersion. Let 
\begin{equation*}
\{X_1, X_2 \} =\left \{\dif \pi_p\left(\frac{1}{\lambda}\, \e{1}|_p\right), \ \dif \pi_p \left(\frac{1}{\lambda} \,\e{2}|_p\right)\right \}   
\end{equation*}
be an orthogonal basis for $T_{\pi(p)} \mathbb{V}^2$ and let $\{ X_3, X_4, X_5\}$ be an orthogonal frame in the normal bundle of $\widetilde{\eta}$ such that
\begin{equation*}
X_3 \circ \pi|_S = \e{3}|_S,  \ X_4 \circ \pi|_S = \e{4}|_S, \ X_5 \circ \pi|_S = \mathbf{g}|_S. 
\end{equation*}
The second fundamental form $\B _{\pi(p)}$ of $\widetilde{\eta}$ is determined by the bilinear forms 
\begin{equation*}
\mathrm{II}^{\alpha}_{\widetilde{\eta}}:=-\langle \dif \widetilde{\eta}, \, \dif X_{\alpha} \rangle_{\pi(p)}, \quad \alpha=3,4,5.
\end{equation*}
From \eqref{eq:structureHiper} it is easy to see that
\begin{equation*}
\begin{aligned}
\mathrm{II}^{3}_{\widetilde{\eta}}(X_i \otimes X_j)&=-\frac{1}{\lambda^2} \langle \dif \eta, \, \dif \e{3} \rangle \left(\e{i}|_p \otimes \e{j}|_p\right), \quad i,j=1,2;\\
&=\frac{1}{\lambda}\left( \omega^1 \otimes \omega^1_3 - \omega^2 \otimes \omega^2_3\right)\left(\e{i}|_p \otimes \e{j}|_p\right);\\
\mathrm{II}^{4}_{\widetilde{\eta}}(X_i \otimes X_j)&=\frac{1}{\lambda}\left( \omega^1 \otimes \omega^1_4 - \omega^2 \otimes \omega^2_4\right)\left(\e{i}|_p \otimes \e{j}|_p\right);\\
\mathrm{II}^{5}_{\widetilde{\eta}}(X_i \otimes X_j)&=\frac{1}{\lambda}\left( \omega^1 \otimes \omega^1 - \omega^2 \otimes \omega^2\right)\left(\e{i}|_p \otimes \e{j}|_p\right).\\
\end{aligned}
\end{equation*}
Denote by $\widetilde{A}^{\alpha}$ the shape operators of $\widetilde{\eta}$ at $\pi(p)$, then

\begin{equation*}
\widetilde{A}^3=\frac{1}{\lambda(p)}\left[
\begin{array}{rr}
\omega^1_3(\e{1}) & \omega^1_3(\e{2})\\
-\omega^2_3(\e{1}) & -\omega^2_3(\e{2})
\end{array}
\right], \quad
\widetilde{A}^4=\frac{1}{\lambda(p)}\left[
\begin{array}{rr}
\omega^1_4(\e{1}) & \omega^1_4(\e{2})\\
-\omega^2_4(\e{1}) & -\omega^2_4(\e{2})
\end{array}
\right],   
\end{equation*}
 
\begin{equation*}
\widetilde{A}^5=\frac{1}{\lambda(p)}\left[
\begin{array}{rr}
1 & 0\\
0 & -1
\end{array}
\right]. 
\end{equation*}
Hence, $\widetilde{\eta}$ is a minimal immersion if and only if
\begin{equation}\label{eq:minimal_1}
\left\{
\begin{array}{l}
\omega^1_3(\e{1})-\omega^2_3(\e{2})=0\\
\omega^1_4(\e{1})-\omega^2_4(\e{2})=0.
\end{array}
\right. 
\end{equation} 
By taking exterior differentiations of \eqref{eq:lambda} we have
\begin{equation*}
\dif \omega^5_i = \dif \lambda_i \wedge \omega^i - \lambda_i \, \omega^i_k \wedge \omega^k.
\end{equation*} 
From Codazzi equation given in \eqref{eq:G-C_Hyper} it follows that
\begin{equation}\label{eq:dif_lambda}
\dif \lambda_i \wedge \omega^i +(\lambda_k-\lambda_i) \,\omega^i_k \wedge \omega^k=0, \quad i =1,2.
\end{equation}
Evaluating these equation on $\e{k} \otimes \e{i}$ we obtain
\begin{equation}\label{eq:dif_lambda12}
\left\{
\begin{array}{l}
\e{1}[\lambda]-2\lambda \, \omega^1_2(\e{2})=0\\
\e{2}[\lambda]+2\lambda \, \omega^1_2(\e{1})=0
\end{array},
\right.
\end{equation}
\begin{equation}\label{eq:dif_lambda34}
\left\{
\begin{array}{l}
\e{3}[\lambda]+\lambda \, \omega^1_3(\e{1})=0\\
\e{3}[\lambda]+\lambda \, \omega^2_3(\e{2})=0
\end{array}\right., \quad
\left\{
\begin{array}{l}
\e{4}[\lambda]+\lambda \, \omega^1_4(\e{1})=0\\
\e{4}[\lambda]+\lambda \, \omega^2_4(\e{2})=0
\end{array}
\right. ,
\end{equation} 
\begin{equation}\label{eq:dif_lambda34_B}
\left\{
\begin{array}{l}
2\omega^1_2(\e{3})+\omega^2_3(\e{1})=0\\
2\omega^1_2(\e{4})+\omega^2_4(\e{1})=0
\end{array}\right. .
\end{equation} 
Therefore, the equations \eqref{eq:dif_lambda34} imply that the condition \eqref{eq:minimal_1} is satisfied. Thus,

%Define the functions 
%\begin{equation}\label{eq:gf}
%\begin{aligned}
%f_3&:=\omega^3_2(\e{1}), \quad f_4:=\omega^4_2(\e{1})\\
%g_3&:=\omega^3_2(\e{2}), \quad g_4:=\omega^4_2(\e{2}).
%\end{aligned}
%\end{equation}Using \eqref{eq:defgf} and \eqref{eq:FC1} we may write

\begin{equation*}
\widetilde{A}^3=\frac{1}{\lambda(p)}\left[
\begin{array}{rc}
-g_3 & f_3\\
f_3 & g_3
\end{array}
\right], \quad
\widetilde{A}^4=\frac{1}{\lambda(p)}\left[
\begin{array}{rc}
-g_4 & f_4\\
f_4 & g_4
\end{array}
\right],   
\end{equation*}
 
\begin{equation*}
\widetilde{A}^5=\frac{1}{\lambda(p)}\left[
\begin{array}{rr}
1 & 0\\
0 & -1
\end{array}
\right]. 
\end{equation*}
From  \eqref{eq:FC1}, \eqref{eq:dif_lambda12}, \eqref{eq:dif_lambda34} and \eqref{eq:dif_lambda34_B}, we have the skew-symmetric matrices 
\begin{equation*}
\left(\omega^i_j(\e{1})\right)=\left[
\begin{array}{cccc}
0 & -\frac{1}{2}\e{2}[\ln \lambda] & -g_3 & -g_4\\
* & 0 & -f_3 & -f_4\\
* & * & 0 & \omega^3_4(\e{1})\\
* & * & * & 0\\
\end{array}
\right], 
\end{equation*}

\begin{equation*}
\left(\omega^i_j(\e{2})\right)=\left[
\begin{array}{cccc}
0 & \frac{1}{2}\e{1}[\ln \lambda] & f_3 & f_4\\
* & 0 & -g_3 & -g_4\\
* & * & 0 & \omega^3_4(\e{2})\\
* & * & * & 0\\
\end{array}
\right], 
\end{equation*} 
 
\begin{equation*}
\left(\omega^i_j(\e{3})\right)=\left[
\begin{array}{cccc}
0 & \frac{f_3}{2} & 0 & 0\\
* & 0 & 0 & 0\\
* & * & 0 & \omega^3_4(\e{3})\\
* & * & * & 0\\
\end{array}
\right], \quad
\left(\omega^i_j(\e{4})\right)=\left[
\begin{array}{cccc}
0 & \frac{f_4}{2} & 0 & 0\\
* & 0 & 0 & 0\\
* & * & 0 & \omega^3_4(\e{4})\\
* & * & * & 0\\
\end{array}
\right],
\end{equation*}  
and  relations
\begin{equation}
\left\{
\begin{aligned}
\e{3}[\lambda]&=\lambda g_3\\
\e{4}[\lambda]&=\lambda g_4 .
\end{aligned}
\right.
\end{equation}

Using the equations given in \eqref{eq:KG_KN} we have that $\K$ and $\KN$ of the minimal immersion $\widetilde{\eta} \! : \mathbb{V}^2 \rightarrow \mathbb{S}^5$ are given by
\begin{equation}\label{eq:Gauss_Immersion1}
\K(\pi(p))=1-\frac{1}{\lambda^2}\left(1+f_3^2+f_4^2+g_3^2+g_4^2 \right),
\end{equation}

\begin{equation}\label{eq:Gauss_Immersion2}
R^3_{412}=\frac{2}{\lambda^2}\left(f_3g_4-f_4g_3\right), \quad R^3_{512}=-\frac{2f_3}{\lambda^2}, \quad R^4_{512}=-\frac{2f_4}{\lambda^2},
\end{equation}

\begin{equation}\label{eq:Normal_Immersion}
\KN(\pi(p))=\frac{16}{\lambda^4}\left(f_3^2+f_4^2+(f_3g_4-f_4g_3)^2\right).
\end{equation} 
By a straightforward calculation it follows that
\begin{equation}\label{eq:bracket}
\begin{aligned}
\left[\e{1},\e{2}\right]&=-\frac{1}{2}\e{2}[\ln \lambda]\e{1}+\frac{1}{2}\e{1}[\ln \lambda]\e{2}+2 f_3 \e{3} + 2 f_4 \e{4},\\
[\e{3},\e{1}]&= \ \ \ \ g_3 \e{1} + \frac{f_3}{2} \e{2} + \omega^3_4(\e{1})\e{4},\\
[\e{3},\e{2}]&=-\frac{f_3}{2} \e{1} +g_3 \e{2} + \omega^3_4(\e{2})\e{4},\\
[\e{4},\e{1}]&= \ \ \ \ g_4 \e{1} + \frac{f_4}{2} \e{2} - \omega^3_4(\e{1})\e{3},\\
[\e{4},\e{2}]&=-\frac{f_4}{2} \e{1} +g_4 \e{2} - \omega^3_4(\e{2})\e{3},\\
[\e{3},\e{4}]&=  \ \omega^3_4(\e{3})\e{3}+\omega^3_4(\e{4})\e{4}.
\end{aligned}
\end{equation} 
Applying the relation between the bracket operation on vector fields and the exterior  
differentiation of 1-form, we obtain the derivatives of the functions $f_i$ and $g_i$:
\begin{equation}\label{eq:difI}
\begin{aligned}
\e{1}[f_3]&=-\e{2}[g_3]-\omega^3_4(\e{1})f_4 - \omega^3_4(\e{2}) g_4\\
\e{2}[f_3]&=\ \ \, \e{1}[g_3]+\omega^3_4(\e{1})g_4 - \omega^3_4(\e{2}) f_4\\
\e{1}[f_4]&= -\e{2}[g_4]+\omega^3_4(\e{1})f_3 + \omega^3_4(\e{2}) g_3\\
\e{2}[f_4]&=\ \ \, \e{1}[g_4]-\omega^3_4(\e{1})g_3 + \omega^3_4(\e{2}) f_3;
\end{aligned}
\end{equation}  
 
\begin{equation}\label{eq:difII}
\begin{aligned}
\e{3}[f_3]&=2 f_3 g_3 - \omega^3_4(\e{3}) f_4\\
\e{3}[f_4]&=f_3 g_4 + f_4 g_3 + \omega^3_4(\e{3}) f_3\\
\e{3}[g_3]&= g_3^2 - f_3^2 + 1 - \omega^3_4(\e{3}) g_4\\
\e{3}[g_4]&=g_3 g_4 - f_3 f_4 + \omega^3_4(\e{3}) g_3;
\end{aligned}
\end{equation}   
 
\begin{equation}\label{eq:difIII}
\begin{aligned}
\e{4}[f_3]&=f_3 g_4 + f_4 g_3 - \omega^3_4(\e{4}) f_4\\
\e{4}[f_4]&=2 f_4 g_4 + \omega^3_4(\e{4}) f_3\\
\e{4}[g_3]&=g_3 g_4 - f_3 f_4 - \omega^3_4(\e{4}) g_4\\
\e{4}[g_4]&= g_4^2 - f_4^2 + 1 + \omega^3_4(\e{4}) g_3.
\end{aligned}
\end{equation} 
By a straightforward calculation, using \eqref{eq:difII} and \eqref{eq:difIII}, we have that
\begin{equation}\label{eq:dif_KN}
\begin{aligned}
\e{3}\left[f_3^2+f_4^2+g_3^2+g_4^2\right]&=2g_3\left(f_3^2+f_4^2+g_3^2+g_4^2+1\right)\\
\e{4}\left[f_3^2+f_4^2+g_3^2+g_4^2\right]&=2g_4\left(f_3^2+f_4^2+g_3^2+g_4^2+1\right)\\
\e{3}[R^3_{412}]&=R^4_{512}\\
\e{4}[R^3_{412}]&=-R^3_{512}\\
\e{3}[R^3_{512}]&=-\omega^3_4(\e{3}) R^4_{512}\\
\e{4}[R^3_{512}]&=R^3_{412}-\omega^3_4(\e{4})R^4_{512}\\
\e{3}[R^4_{512}]&=-R^3_{412}+\omega^3_4(\e{3}) R^3_{512}\\
\e{4}[R^4_{512}]&=\omega^3_4(\e{4}) R^3_{512}.
\end{aligned}
\end{equation} 
Using the above equations it follows that
\begin{equation*} 
\e{3}[\K]=0, \quad \e{4}[\K]=0, \quad \e{3}[\KN]=0, \quad \e{4}[\KN]=0.
\end{equation*}
Therefore, the functions $\K$ and $\KN$ are constant on  $\mathcal{F}^2$, which is according with the fact that $\tilde{\eta}$ is well defined.

By Theorem \ref{Teorema1} \ $\mathbf{x_{\boldsymbol{\widetilde{\eta}}}}:\mathcal{N}_* \rightarrow \mathbb{S}^5$ is an immersed minimal hypersurface of $\mathbb{S}^5$ with zero Gauss-Kronecker and $3$-mean curvatures. 

Define the map $\tau \! : M^4_{*} \to \mathcal{N}_{*}$ by $\tau(p)=(\pi(p), \mathbf{g}(p)).$
Thus   $\mathbf{x_{\boldsymbol{\widetilde{\eta}}}} \circ \tau (p)= \mathbf{g}(p).$   Since, $\pi=\pi_1 \circ \tau$ we have that $\eta|_{M^4_{*}}=\widetilde{\eta} \circ \pi_1 \circ \tau.$  Finally, the metric on $\mathcal{N}_{*}$ is induced by $\mathbf{x_{\boldsymbol{\widetilde{\eta}}}}$,
then $\tau$ must be a local isometry. This proves the theorem.
\end{proof}

%%%%%%%%%%%%%%%%%%%%%%%%%%%%%%%%%%%%%%%%%%%%%%%%%%%%%%%%%%%%%
%%%%%%%%%%%%%%%%%%%%%%%%%%%%%%%%%%%%%%%%%%%%%%%%%%%%%%%%%%%%%
%%%%%%%%%%%%%%%%%%%    Bibliografia      %%%%%%%%%%%%%%%%%%%%
%%%%%%%%%%%%%%%%%%%%%%%%%%%%%%%%%%%%%%%%%%%%%%%%%%%%%%%%%%%%%
%%%%%%%%%%%%%%%%%%%%%%%%%%%%%%%%%%%%%%%%%%%%%%%%%%%%%%%%%%%%%

\nocite{Ramanathan.1990}
\nocite{Hasanis.2007}

\bibliographystyle{amsplain} %alpha, plain, unsrt, abbrv, amsplain, amsalpha
\bibliography{BiblGeometria}

\end{document}